\pdfoutput=1
\documentclass[oneside,english]{amsart}
\usepackage[utf8]{inputenc}
\usepackage[a4paper]{geometry}
\usepackage{color}
\usepackage{refstyle}
\usepackage{amstext}
\usepackage{amsthm}
\usepackage[foot]{amsaddr}
\usepackage{amssymb}
\usepackage{thmtools}
\usepackage{thm-restate}
\usepackage{csquotes}
\usepackage{nicefrac}
\usepackage{enumitem}
\usepackage[
backend=biber,
style=alphabetic,
citestyle=alphabetic
]{biblatex}
\makeatletter

\AtBeginDocument{\providecommand\lemref[1]{\ref{lem:#1}}}
\AtBeginDocument{}
\AtBeginDocument{\providecommand\defref[1]{\ref{def:#1}}}
\RS@ifundefined{subsecref}
 {\newref{subsec}{name = \RSsectxt}}
 {}
\RS@ifundefined{thmref}
 {\def\RSthmtxt{theorem~}\newref{thm}{name = \RSthmtxt}}
 {}
\RS@ifundefined{lemref}
 {\def\RSlemtxt{lemma~}\newref{lem}{name = \RSlemtxt}}
 {}

\numberwithin{equation}{section}
\numberwithin{figure}{section}
\theoremstyle{plain}
\newtheorem*{fac*}{\protect\factname}
\theoremstyle{remark}
\newtheorem*{rem*}{\protect\remarkname}
\theoremstyle{plain}
\newtheorem{thm}{\protect\theoremname}[section]
\theoremstyle{definition}
\newtheorem{defn}[thm]{\protect\definitionname}
\theoremstyle{plain}
\newtheorem{lem}[thm]{\protect\lemmaname}
\theoremstyle{remark}

\theoremstyle{plain}
\newtheorem{fact}[thm]{\protect\factname}
\theoremstyle{plain}
\newtheorem{cor}[thm]{\protect\corollaryname}
\theoremstyle{plain}
\newtheorem{conjecture}[thm]{\protect\conjecturename}
\theoremstyle{plain}
\theoremstyle{plain}
\usepackage{refstyle}
\usepackage{url}
\usepackage{relsize}
\newref{definition}{name=Definition~,Name=Definition~,names=Definitions~,Names=Definitions~}
\newref{defn}{name=Definition~,Name=Definition~,names=Definitions~,Names=Definitions~}
\newref{claim}{name=Claim~,Name=Claim~,names=Claims~,Names=Claims~}
\newref{theorem}{name=Theorem~,Name=Theorem~,names=Theorems~,Names=Theorems~}
\usepackage{bbm}

\makeatother

\usepackage{babel}
\providecommand{\claimname}{Claim}
\providecommand{\corollaryname}{Corollary}
\providecommand{\conjecturename}{Conjecture}
\providecommand{\factname}{Fact}
\providecommand{\lemmaname}{Lemma}
\providecommand{\remarkname}{Remark}
\providecommand{\definitionname}{Definition}
\providecommand{\theoremname}{Theorem}
\usepackage{hyperref}
\usepackage{cleveref}

\begin{document}

\title{ High dimensional expansion using zig-zag product \\ \small
	(Working Draft) }

	\author{ $\text{Eyal Karni}^\dagger$  $\text{Tali Kaufman}^{\#}$  }
\address{This work was conducted as part of a master's dissertation in mathematics}
\address{for Bar Ilan University, ISRAEL}

\address{$\dagger$ Bar Ilan University, ISRAEL. Email: eyalk5@gmail.com} 
\address{\# Bar Ilan University, ISRAEL. Email:  kaufmant@mit.edu} 
\begin{abstract}
	We wish to renew the discussion over recent combinatorial structures that are 3-uniform hypergraph expanders,
	viewing them in a more general perspective, shedding light on a previously unknown relation to the zig-zag product.
	We do so by introducing a new structure called triplet structure, that maintains the same local environment around each vertex.
	The structure is expected to yield, in some cases, a bounded family of hypergraph expanders whose 2-dimensional random walk converges.

	We have applied the results obtained here to several known constructions, obtaining better expansion rate than previously known. Namely, we did so in the case of Conlon's construction and the $S=[1,1,0]$ construction by Chapman, Linal and Peled. 
	
\end{abstract}

\date{\today}

\maketitle
\section{Introduction}

The notion of expansion is a key one in graphs, having many applications to a variety of mathematical problems. It generally means how much easy is it to get from one vertex set to another in a graph. Among its applications are practical such as error correcting codes \cite{sipser1994expander} which could be used in communication, and theoretical as it was used to prove the PCP theorem \cite{dinur2007pcp}.

Over the last two decades or so, there have been various attempts to generalize this notion to higher dimensions. That means to talk about expansion in hypergraphs. There has been a growing interest in this field, motivated partially by its usefulness to constructing quantum error correcting codes. But it was speculated that it would help to solve some theoretical questions that are out of reach for expander graphs otherwise (LTC\footnote{Local Testable Codes} as an example \cite{kaufman2013high}).

The various attempts used various definitions for the expansion, in contrary to 1-dimensional (or graph) case.  We will mainly deal with the definition of $2D$-random walk convergence that was originally appeared in \cite{kaufman2016high}.  In the simple case of 3-uniform expander, this is essentially the random walk convergence rate of the random walk in the auxiliary graph (definition \ref{def:aux}).  Unfortunately, it seems that the auxiliary graph can't be such a good expander, as we have reasons\footnote{This is true in case the links are good enough expander} to believe that the convergence rate is bounded from below by $\nicefrac{2}{3}$ \cite{kaufman2017high}.  This notion also relates to agreement expanders, as appears in \cite{dinur2017high}.

As it is generally difficult, the construction of bounded-degree HDE is typically explicit and based on heavy algebraic tools. Explicit means that there is a concrete family of hypergraphs that satisfy the required expansion. Until recently, the main construction which is bounded-degree and satisfies strong HDE\footnote{High Dimensional Expansion} properties($2D$-random walk convergence but also some stronger notions), arises from Ramanujan complexes \cite{lubotzky1988ramanujan}. 

The innovation of David Conlon in his construction \cite{conlon2017hypergraph} was that he used combinatoric method to allow one to take random set of generators of Cayley graphs and to provide an 3-dimensional hypergraph built upon this that satisfies some higher-dimensional expansion properties including the geometric-overlapping property and the 2D random walk.  The drawback of this construction is that it is built upon abliean groups.

Abelian groups are known to have poor expansion properties. Indeed, Alon-Roichman's paper \cite[Proposition 3]{alon1994random} suggests that in order to achieve $\epsilon$-expansion for abliean family $G$ of order $n$, one needs at least $C\ \log\ n$ generators. Thus, its construction couldn't be of bounded degree.

Recently, some progress in constructing bounded-degree HDE in combinatorial ways has been made.  \cite{chapman2018expander} studied regular graphs whose links are regular.  They achieved some bounds concerning the maximal non-trivial eigenvalue for such graphs. And they constructed a hypergraph expander from expander graph whose $2D$-random walk converges rapidly\footnote{The graph should be of high girth}.

Even more recently, \cite{liu2019high} provided a construction that is based on any graph and achieves constant spectral gap on any high order random walk.
We will just note that this work was done independently of their work.

These structures share some common properties. Particularly, they use graph product of good expander graphs, in order to achieve the sort of commutativity that lies in the heart of Conlon's structure.  This typically\footnote{Some other cycles where considered in \cite{chapman2018expander} } results in some 4-cycles or commutative generators (in case of Cayley graphs) that are used to provide expansion as they are shared in several triplets (and centers). 

We believe the triplet structure we suggest here encompasses the essence of this phenomenon.  Basically, we look at a family of such constructions that involves applying a certain local 3-dimensional structure around each vertex, in a way similar to the way Cayley graph applies the same set of generators to each vertex. This results in edges that are shared between centers, and ultimately allows us to achieve 2-dimensional random walk convergence depending on inherited properties of the group and the local construction. We were able to achieve this, for example, in all bounded cases of this construction.

In contrary to the traditional method of using links, we have analyzed the convergence rate of the hypergraph in our construction by looking at the auxiliary graph of the random walk(definition \ref{def:aux}), and relating it to other equivalent graphs, and finally to the previously known zig-zag product.  And so, as an application of our results, we were able to achieve a better convergence rate\footnote{than achieved in the original papers} in some known constructions. 

\medskip \medskip Namely, in the following constructions: 
\begin{enumerate}
	\item Conlon's construction
	   
	\item The 3-product construction, which is a slight adaptation of \cite{chapman2018expander}
	 in the case $[1,1,0]$
\end{enumerate}

Possibly, other expansion properties of these constructions could be investigated more easily by this approach.
And possibly some other interesting cases could be studied. 

\section{Preliminaries}
We will introduce some general definitions first, and then introduce our structure. 
\subsection{Basic definitions }
We introduce here some  basic definitions. But we do assume some familiarity with graphs and expanders.
We refer the reader to \cite{hoory2006expander} for further knowledge about expanders, lifts and the zig-zag product. 
\begin{defn}
Let $d=\lambda_{1}\ge\lambda_{2}...\ge\lambda_{n}$ be the eigenvalues
of $A$, the adjacency matrix of $G$ 
\[
\lambda(G):=max_{\lambda_{i}\neq\lambda_{1}}|\lambda_{i}(A)|
\]
$\lambda$ is the maximal non-trivial eigenvalue. \\
\label{definition:A-family-of}\par A family of graphs $\{G_\alpha\}_\alpha$ 
	is (at-least) $\epsilon$-expanding if 
\[
	\forall G\in \{G_\alpha\}_\alpha\ \epsilon \le 1-\lambda(G) 
\]
\end{defn}

\begin{defn}
\label{definition:CayleyGiven-a-group}Given a group G, with a set of
generators $S$, satisfying $S=S^{-1}$, the Cayley graph $Cay(G,S)$
is the undirected graph with vertex set G and edge set \{$(sg,g)\ |\ g\in G\ s\in S$\}.
\end{defn}

The Cayley graph allows one to examine the abstract structure of a
group as a graph. For certain groups, it has excellent expansion properties.

\begin{defn}
$H=(V,E)$ is an\emph{ hypergraph} where $E\subseteq P(V)$\footnote{Each hyperedge is a subset of V} 

$V$ are called the \emph{vertices} and $E$ are the \emph{hyperedges}. 
We will only look at 3-uniform hypergraphs here (where all hyperedges have the same size $3$).

We also have the following related definitions: 

\begin{itemize} 
\item $T(H)=\{ \tilde{t} \in E\ \mid \ |\tilde{t}|=3\}$ is the set of triplets (or hyperedges).
\item $E^{1}(H)$ is the collection of subsets of size 2 of every in $T(H)$. 
\item The vertices, together with the edges in $E^{1}$ form a graph known as the 1-skeleton of H
\item We will refrain from using the word edges for $e\in E$ and instead
use it for $\mathcal{ E }\in E^{1}(H)$ 

($\mathcal{ E }$ specifies a general element in $E^1$)

\item The degree of a vertex is the number of hyperedges that contains it.
\item A hypergraph is $k-\emph{regular}$ if every vertex has degree $k$.
\item A hypergraph is $k-\emph{edge regular}$ if every edge ($\mathcal{E} \in E^{1}$) is contained in exactly $k$ triplets. 
\end{itemize}

\end{defn}

\begin{defn}
	\emph{$2D$-random walk} on a 3-uniform hypergraph is defined to be a
	sequence of edges $\mathcal{ E }_{0},\mathcal{ E }_{1},\cdot\cdot\cdot\in E^1(H)$ such that\footnote{These definitions were taken from \cite{conlon2017hypergraph} By Conlon.
		Original def. from \cite{kaufman2016high} by Kaufman and Mass}
\begin{enumerate}
	\item $\mathcal{ E }_{0}$ is chosen in some initial probability distribution $p_{0}$
	 on $E^{1}$(H)
	\item for every $i>0$, $\mathcal{E}_{i}$ is chosen uniformly from the neighbors of
	 $\mathcal{E}_{i-1}$. That is the set of $f\in E^{1}$ s.t. $\mathcal{E}_{i-1}\cup f$
	 is in $T$
\end{enumerate}

\end{defn}
\subsection{\label{subsec:Structures}Triplet structure}

We will need several definitions first.

Let $G$ be a group\footnote{The mentioned notation conventions will be used implicitly throughout the paper, unless defined otherwise} ($g$ specifies a general element in $G$, but also $a,b,c$ or $s,s',s''$).

Let $\mathcal{S}\subset{G \choose 3}$ be a set of triples ($\tilde{s}$
specify a general element in $\mathcal{S}$). 
\begin{defn}
	$\mathcal{T}={\mathcal{S} \choose 2}$ is the set of types ($\tau=\{\tau_{1},\tau_{2}\}$
	is an element in $\mathcal{T}$). \\
We have the ordered version of it 
\[\mathcal{T}_{o} = \text{\{the ordered
2-subsets of $\mathcal{S}$\}}\] 
    ($t=(t_{1},t_{2})$ is an element in $\mathcal{T}_{o}$)

\end{defn}
\begin{defn}We define hypergraph $H=St(G,\mathcal{S})$ by specifying its triples , using the mapping:
	\[St:G\times\mathcal{S}\rightarrow T(H)\]
	\[(g,\{s,s',s''\})\rightarrow\{sg,s'g,s''g\}\] 

	$T(H)$ is the set of triples in $H$($\tilde{t}$ is a general element
in $T(H)$ ).

$St$ can be seen as a product:  \[G \times \mathcal{S} \rightarrow T(H)\]
And so,
\[
	T(H)=\{\tilde{s}g\ |\tilde{s}\in \mathcal{S}, g\in G \}
\]

A hypergraph defined in this way, and satisfies additional conditions(in \ref{sub:conditions}) will be called a \textbf{triplet structure}.

\end{defn}
\begin{defn}
	The 1-skeleton of $H$ is
 \[
     E^{1}:=\{\{ u,v \}\ |\ \{ u,v,w \}\in T\ \text{for some }w\in V(H)\}
	\]
	
	But it could also be described as
 \[
	 E^{1}= \{ \{ sg,s'g \}\ |\ \{s,s' \} \in\mathcal{T},\ g\in G\ \} 
	\]
\end{defn}

While it is tempting to look at the type of edge by the way upon which it acts on a vertex on the 1-skeleton (i.e. 
$(sg,s'g)\text{ is of type }s^{-1}s'$), 
viewing an edge $(ag,bg)\in E^{1}$ as of type $\{a,b\}$
in center $g$ would be more useful to us.
This is why we define a function $e$.

\begin{defn}
	$e$ gets the edge by center and type 
	 \[e:G\times\mathcal{T}\rightarrow E^{1}\]
	\[e(g,\tau)=\{\tau_{1}g,\tau_{2}g\}\]

	And similarly \[e_{o}:G\times\mathcal{T}_{o}\rightarrow E_{o}^{1}\]
	is defined by $e_{o}(g,t)=(t_{1}g,t_{2}g)$ (the ordered skeleton edges).
\end{defn}

\begin{defn}
 \label{ndef}
	An edge $w$ is in center $c$ if \[\exists\tau\text{ s.t }e(c,\tau)=w\]
	 \\ The set of centers of $w$ will be a function $c:E^{1}\rightarrow P(G)$
\end{defn}

We have a natural inverse in $\mathcal{T}_{o}$. That is
\[
	(a,b)^{-1}:=(b^{-1},a^{-1})
\]
And for every $\tau={ \tau_1,\tau_2 }$ such that $\tau_1 \tau_2 = \tau_2 \tau_1$, 
we have a natural product  
\[\tau \cdot g = \tau_1\tau_2 g\]
We need another definition here. \\
	\par When viewing $\mathcal{S}$ as a 3-uniform hypergraph, we have that $L$ is the edge-triple adjacency graph. 

	\begin{defn}\label{graphl} The graph $L$ is defined by 
 
	\[V(L)=\mathcal{T}\] \[\tau\sim\tau'\text{ }\iff\text{ }\exists\tilde{s}\in\mathcal{S}\text{ s.t. }\tau,\tau'\subset\tilde{s}\]
\end{defn}

\subsection{Conditions}%
\label{sub:conditions}

So that we can work with the structure, we will need several conditions.
\begin{enumerate}
	\item[Condition \textbf{0}] \label{itemZ} $\mathcal{T}$ doesn't contain an element $\{s,s^{-1}\}$
	\item[Condition \textbf{A}] \label{itemA} $\mathcal{T}_{o}$ should be symmetric.
		That is if $(a,b)\in\mathcal{T}_{0}$ then $(a,b)^{-1}\in\mathcal{T}_{o}$. \\
	 This condition is needed for the hypergraph to have a skeleton that is undirected.
 \item[Condition \textbf{B}] \label{itm:itemB} $E^{1}$ contains no non-trivial 4-cycles that is other than the ones arise by the commutative relation if exists, 
	 that is for $t \neq t' \in \mathcal{T}_o$, s.t.
	 \[t_1 t^{ -1 }_2=t'_1(t'_2)^{-1} \implies\] 
	 \[t'_2=t_1^{-1} \newline \text{ and } t'_1=t_2^{-1} \]
	 Another way to put it is that it implies\footnote{with respected to the defined inverse in $\mathcal{T}_o$ }   that $t'=t^{-1}$. We also get that $t_1$ and $t_2$ commutate. 
	 That also assures that $St$ is bijective.
	 Any triplet has just one description in terms of $G,\mathcal{S}$ \footnote{ 
	 Indeed, if \[t_1g,t_2g=t'_1g',t'_2g'\] such that $g'\neq g$  then \[t_1 t^{ -1 }_2=t'_1(t'_2)^{-1} \] 
     and it leads to $g'=t^{-1}g$. But it can also be done with $t_1g,t_3g$.   }
     \item[Condition \textbf{C}] \label{itemC} \[\forall\tau \in \mathcal{T}\ \#\{c\in {\mathcal{S}\choose{1}}  \ s.t.\ \tau\cup c\in\mathcal{S}\}=\tilde{d}\] That is also to say that $H$ is edge regular.
 \item[Condition \textbf{D}] \label{itemD} $L$ should be connected (otherwise, we have no hope of achieving expansion) 
\end{enumerate}

\begin{enumerate}
	\item [Condition \textbf{E}] \label{itemE} for every $\tau=\{a,b\}$, we have $a,b$ commute.
\end{enumerate}

\medskip  \medskip A structure that satisfies all the mentioned conditions will be called \textbf{commutative triplet}.

We will noted such constructions by 
\[H=Cts( G,\mathcal{S} )\]

\begin{rem*}
     \medskip  \medskip Any $H$ that satisfies all  conditions $0-D$ above will be called a \textbf{triplet structure}.    We will not talk about triplet structures here in general, but will mention them in the concluding remarks.
\end{rem*}


\medskip  \medskip For random walk on $H$, we define an auxiliary graph $G_{Walk}$:
\begin{defn} \label{def:aux} The graph $G_{walk}$ is defined by 
	
	\[V(G_{Walk})=E^1(H)\]
	\[\mathcal{ E }\sim\ \mathcal{ E }'\text{ }\iff\text{ }\exists\ \tilde{t}\in\ T(H)\text{ s.t. }\mathcal{ E },\mathcal{ E }'\in\ \tilde{t} \]
\end{defn}

We define $G_{cay}$ to be $Cay(G,\mathcal{ T })$
\[
\{ (g,\tau g)\ |\ \tau \in \mathcal{T} \} 
\]
\begin{rem*}
    Here we abuse notation.

    $\tau=\{a,b\}$ $\tau\cdot g$ corresponds to $abg=bag$ 
\end{rem*}
Notice that since $\mathcal{T}$ is symmetric, this graph is well-defined and undirected.
\section{Overview of the paper}
\subsection{Main Construction}%
\label{ssub:overview_of_the_paper}

\medskip \medskip Our main theorem suggest that the expansion property of $H=Cst(G,\mathcal{ S })$ could be deduced from the expansion
properties of both $G_{ cay }$ and $L$.  $L$ is simply the 2-dimensional random walk convergence on $\mathcal{S}$, that is the local structure. And $G_{ cay}$ is a Cayley graph generated by the two subsets of $\mathcal{S}$. 

To relate them, we do several steps: 

\begin{enumerate}[itemsep=9pt]
    \item 
	We consider the random walk on the hypergraph $H$ as a random walk on the auxiliary graph $G_{walk}$.
    \item  $G_{rep}:=G_{cay}\raisebox{.5pt}{\textcircled{\raisebox{-.9pt}{r}}}L$ (definition \ref{def:We-define-graph})

	\medskip \medskip We build the replacement graph of $G_{cay}$ and $L$, that will be called $G_{rep}$. A vertex in this graph can be thought of as an alternative description of an edge. That is being in a center $c \in V(G_{cay})$ and of a certain type $\tau \in V(L)$. 
    \item Translating the random walk to the alternative description (section \ref{mainpart})

	\medskip \medskip The original random walk can be translated to a random walk in the alternative description \footnote{The proof just use this fact as an intuition }. In this description we are doing a random walk on the graph $G_{rep}$. It is not a standard random walk, but instead it is done by operator $T$ over the graph(to be defined).

\medskip \medskip 	We notice that the edges of $G_{rep}$ can be colored in two colors, where red edges are inside each center, and blue edges relate two different centers. 
	A step in the original random walk mix inside the original center\footnote{The center that "it came from", or that is common to the previous vertex} in probability $\nicefrac{1}{2}$, and in probability $\nicefrac{1}{2}$ mix in another center. The intuition for this is described in \cref{sec:randcon}.
	Therefore, the original random walk is modeled as an \textbf{operator $T$} over the graph $G_{rep}$.
	This operator can be regarded as an operator that in probability $\nicefrac{1}{2}$ picks a red edge, and in probability $\nicefrac{1}{2}$ picks a blue edge and after that a red edge. 
    \item  $ \lambda(G_{walk}) \le \lambda( T )$ (corollary \ref{corlift})

	\medskip \medskip We now regard $T$ as a new graph. We don't relate a random walk on $G_{walk}$ to that on $T$ directly(which is possible), but we rather prove that $T$  is a lift of $G_{walk}$, thus it has at least as good expansion as $G_{walk}$

    \item $ \lambda(T^2) \le \frac{1}{2}+\frac{1}{2}\lambda(G_{cay}\raisebox{.5pt}{\textcircled{\raisebox{-.9pt}{z}}}L)$ (lemma \ref{lem:In-case-there})

	\medskip \medskip The operator $T^2$ picks edges that looks like red blue red in probability $\nicefrac{1}{2}$. Zig-zag product is defined as those edges over a replacement graph. And we know that Zig-zag product is good in terms of expansion. 
\end{enumerate}

\medskip \medskip We get the following result, which is our main theorem.
\begin{restatable*}{thm}{mainthm}
 $\lambda(G_{walk})\le \sqrt{\frac{1}{2}+\frac{1}{2}\lambda(G_{cay}\raisebox{.5pt}{\textcircled{\raisebox{-.9pt}{z}}}L) }$ 
\end{restatable*}

Thus, the expansion properties of $H$ are somehow the middle ground between the expansion properties of the mentioned graphs, in a way similar to that of the Zig-zag product. This is a consequence of this similarity.

\subsection{Applications}%
\label{sub:cases}

\par 
We will apply this to a several interesting cases, to achieve a better results than originally achieved, by more primitive expansion calculation.

\subsubsection{The Construction By Conlon}%
\label{ssub:conlon}

Conlon looked at $Cay(G,S)$ where $S$ is a set of generators with no non-trivial 4-cycles. He built a hypergraph $H$ which is based upon triples of this graph. The triples of $H$ are composed of 3 different vertices adjacent to the same vertex, and were divided to cliques naturally.

\medskip  \medskip We have a set $S\subset G$ s.t. $S=S^{-1}$.\\ We define hypergraph $H$ by its triples:

\[
	T(H)=\{s_{a}g,s_{b}g,s_{c}g\ |s_{a},s_{b},s_{c}\in S\ distinct\ and\ g\in G\}
\]

In our case it is enough to define $\mathcal{S}={S \choose 3}$
and we have that \[H=Cts( G,\mathcal{S} )\]

We require that there are no non-trivial 4-cycles in $Cay(G,S)$
for condition \hyperref[itm:itemB]{\textbf{B}} to be satisfied. Indeed, this is equivalent, because a non-trivial $4$-cycle in $Cay(G,S)$ is
\[
	abcd=e
\]
 \[a,b,c,d\in\ S\text{ s.t. }\{c,d\}\neq\{a,b\}\]
 Which contradicts condition \hyperref[itm:itemB]{\textbf{B}}. \\

So far we have described a small generalization of Conlon's construction. To describe it specifically, we require that $G=\mathbb{F}_{2}^{t}$ and the product
is additive.\\

It is easy to see that the rest of the conditions are satisfied in this case.

We have the following result:

\begin{restatable*}{cor}{conlons}
\label{conlon}
	Assuming
	\[G=\mathbb{F}_{2}^{t}\]
	\[S\subset\ G\text{ such that }\]
	\[a+b=c+d\text{
			only in trivial case }\text{ (}a,b,c,d\in\ S)\]

	Then for \[\mathcal{S}={S \choose 3}\text{ }H=Cts( G,\mathcal{S} )\]

	in terms of the original graph \[\lambda=\lambda(Cay(G,S))\]
	we have that the convergence rate on $2D$-random walk is \[
	 \frac{\sqrt{3}}{2} +\frac{1}{2\sqrt{3}}\lambda^2 + O(\frac{1}{d})
	\] 
\end{restatable*}

\begin{rem*}
	In terms of expansion of the auxiliary graph,
	we get as $\varepsilon\rightarrow 1$, asymptotic behavior of 
\[
	1 - \frac{\sqrt{3}}{2} - \frac{(\varepsilon - 1)^2}{2 \sqrt{3}} 
\]

	compared to 

	\[
		\frac{\epsilon^{4}}{2^{15}}
	\]

	achieved in \cite{conlon2017hypergraph}. 
	This is asymptotically better.
\end{rem*}

\subsubsection{The 3-product case}%

Chapman, Linal and Peled described a construction called Polygraph in the paper \cite{chapman2018expander}. We will describe it very briefly, and refer the reader to the paper for further explanation.

In this construction, one takes a graph $G$ with large enough girth and a multiset of numbers $S$.
And one defines a graph $G_S$ called polygraph. 
\par The vertices of $G_S$ are ${V(G)}^m$ (tensor product) 

Two vertices $(x_1\ldots x_n),(y_1\ldots y_n)$ are adjacent if the collection $(d(x_i,y_i)\ |\ 1\le i\le m )$ is equal as a multiset to $S$,
where $d$ is the distance function on the graph. 

Finally, one takes the cliques complex of this hypergraph $\mathcal{C}_{G_{S}^{(2)}}$.

\begin{restatable}[3-Product-Case]{defn}{thrprodcase}
\label{def:3-product-case} 

	Given $G_1,G_2,G_3$ groups,\\ with $S_{i}\subset G_{i}$
	s.t.
	\begin{enumerate}
		\item $S_{i}=S_{i}^{-1}$
		\item $d:=|S_{i}|=|S_{j}|$
		\item $|G_{i}|=|G_{j}|$
	\end{enumerate}
	We define
	\[
		G:=G_{1}\times G_{2}\times G_{3}
	\]

	\[
		T=\{s_{1}g_{1},s_{2}g_{2},s_{3}g_{3}\ |\ g_{i}\in G_{i},s_{i}\in S_{i}\}
	\]

	and $H$ defined to be the hypergraph generated by these triples.\\

	We can also describe it as a triplet structure by defining \[\mathcal{S}=\{S_{1},S_{2},S_{3}\}\]

	and \[
		H:=Cts( G,\mathcal{S} )
	\]
\end{restatable}

\medskip \medskip  For $S=[1,1,0]$, $\mathcal{C}_{G_{S}^{(2)}}$ is the same as the hypergraph $H$, in the specific case $G_1=G_2=G_3$.
So, we provide a slight generalization of the $[1,1,0]$ case, as we allow taking different base graphs. 
On the other hand, we force all the graphs to be Cayley graph.

\par We will prove later that it satisfies all the conditions.
We have the following theorem:
\begin{restatable}{thm}{thrprod}
	Given $\lambda_{i}=\lambda(Cay(G_{i},S_{i}))$ ordered s.t. $\lambda_{1}\le\lambda_{2}\le\lambda_{3}$
	the random walk on $H$ defined in \defref{3-product-case}, converges
	with rate $\sqrt{\frac{1}{2}+\frac{1}{2}f(\frac{1+2\lambda_{3}}{3},\frac{1}{2})}$
	where $f$ is the zig-zag function\footnote{originally defined in \cite[Theorem 3.2]{reingold2002entropy}} from \cref{equation:zig-zag} \\ ( $f(a,b)\le a+b$, $f<1$ where
 $a,b<1$ ). 
\end{restatable}
\begin{rem*}
 In terms of expansion of the auxiliary graph,
 we get as $\varepsilon\rightarrow 1$, asymptotic behavior of 
 \[
	1 - \sqrt{\frac{22}{24}} - \frac{1-\varepsilon}{\sqrt{33}} 
 \]
 That is asymptotically better compared to the results in
 \cite{chapman2018expander}, which were 
 \[
 \frac{\epsilon^{4}}{8\cdot10^8}
 \]
\end{rem*}
\section{Commutative Triplet structure Properties}
Here we prove some properties of the structure.
\subsection{Basic properties}
\label{sub:basic_properties}
We assume all along a commutative triplet structure $H$.
We look at $L$ defined in \ref{graphl}. 
Note that condition \hyperref[itemC]{\textbf{C}} suggests that $L$ is $2\tilde{d}$ regular because there are $\tilde{d}$ triplets to choose from.
Each introduces 2 distinct edges to choose form\footnote{ Because if $\tau\cup {c}\text{ and }\tau\cup{c'}$ 
both contain $\{a,b\}$, then $\tau=\{a,b\}$.}

As an illustrative example, in case of Conlon,\[L=J(S,2)\] where $J$ is the Johnson graph.
\begin{defn}
	$J(S,n)$ is the Johnson graph

	\begin{center}
		$V(J)={S \choose n}$ and $v\sim v'$ if $|v\cap v'|=n-1$
		\par\end{center}

		For example, in case $n=2$ \\ \[\{a,b\}\sim\{c,d\}\text{ if both sets share one element} \]
\end{defn}

Notice that $L$ is a subgraph of the $J({\mathcal{S} \choose{1} } ,2) $ graph,
because two elements must have an intersection of size 1,
in order to possibly be adjacent. 

%


We can describe a random walk on the hypergraph
as a random walk on types (that is random walk on $L$), and a random
walk on centers. We intend to construct a graph on the types and centers that would reflect this. 

As in the case of the original construction, the commutativity requirement
translates into an edge being in two centers. And being in two centers leads
to expansion properties of the hypergraph as the random walk
progresses\footnote{This intuition was largely inherited from Conlon's talk at a conference by the IIAS (Israel institute for advanced studies) in April 2018. He just didn't go as far.}. 

\begin{lem}
	\label{lem:double}For $H$ triplet structure
	if $e(g,\tau)=e(g',\tau')$ where $g\neq g'$ then
	\[g'=\tau g\] 
	\[
	\tau' = \tau^{-1}
	\](by abusing notation)
	\begin{rem*}
		We implicitly say that $\tau$ is commutative type.\\
		That is for $\tau=\{t_1,t_2\}$ \[t_2=t_1\] and so \[\tau'=\{t_2^{-1}, t_1^{-1} \}\] 
		and 
		\[g'=t_{2}t_{1}g=t_{1}t_{2}g\]
	\end{rem*}
\end{lem}

\begin{proof}
	Let $t$, $t'$ be the corresponding ordered types.

	That without loss of generalization corresponds to $t_{1}g=t'_{1}g'$ and $t_{2}g=t'_{2}g'$. Of course, if $t=t'$ we get a contradiction.
	\par Then \[t_{1}t_{2}^{-1}=(t'_{1})(t'_{2})^{-1}\]
	and by condition \hyperref[itm:itemB]{\textbf{B}}, \[t'_2 = (t_1)^{-1}\] \[t'_1=(t_2)^{-1}\]

	we get that $t_{2}t_{1}=t_{1}t_{2}$.
	By
	\[g'=(t'_{2})^{-1}t_{2}g\]
	we get that $g'=\tau g$ and $\tau'=\tau^{-1}$.

\end{proof}
\begin{lem}
 \label{lem:Every-edge-is}Every edge is in exactly two centers (see definition \ref{ndef} )
 \end{lem}
%

\begin{proof}
	First $|c(e)|\ge1$ by the definition of $E^{1}$.
	But for $t=(a,b)$ by condition \hyperref[itemE]{\textbf{E}}, we can see that
	\[e(g,t)=e(tg,t^{-1})\]

	Since $ag= b^{-1}abg$ and $bg=a^{-1}abg$.
	\medskip  \medskip \newline
	Suppose $|c(e)|>2$,
	\[e(g,\tau)=e(g',\tau')\]
	If exists another $g'',\tau'$ s.t. $e(g'',\tau')=e(g,\tau)$ then by \lemref{double}
		\[g''=\tau g=g'\]

	%
	%
	%

\end{proof}
\begin{defn}
	Let $G_{\text{walk}}$ be the auxiliary graph of the random walk on edges.
	That is \[
		e\ \sim e' \iff \exists \tilde{t}\text{ s.t. }\ e,e'\ \subset \tilde{t}
	\]where $ A_\text{walk} $ is the adjacency matrix
\end{defn}

\begin{lem}
	\label{lem:InGwalk} In $G_{walk}$ on a commutative triplet structure

\begin{enumerate}
	\item $e\ \sim\ e'$ iff exists $c$ s.t. $e=e(c,\tau)$ $e'=e(c,\tau')$ and
	 $\tau\sim^{L}\tau'$
	\item The walk is $4\tilde{d}$ regular
\end{enumerate}
\end{lem}
\begin{proof}\mbox{}\\*
	\begin{enumerate}
		\item Let's check when $e$ is contained in a
		 certain triplet $\tilde{t}$.

		 The triplet $\tilde{t}$ is $\tilde{s}g$ where $\tilde{s}\in\mathcal{S}$.
		 So, any $2$-subset of it is guaranteed to be of the form $\{ s_{a}g,s_{b}g \}$
		. That means that $g\in c(e)$ is essential condition.

		 If $e\sim e'$, we must have that $c=c(e)\cap c(e')$
   exists, so we define $e'=e_{o}(c,t')$ and
   $e=e_{o}(c,t)$. \\

		 There can't be 2 such centers in the intersection.
		 Suppose they both belong to centers $c,c'$, then $c'=tc$ from \lemref{double}
		 and similarly $c'=t'c$. 

   We define $\tau=\{t_1,t_2\}$ and similarly for $\tau'$. 

		 \begin{center}
			 $\mathcal{ E }=\{\tau_{1}c,\tau_{2}c\}$ $\mathcal{ E }'=\{\tau'_{1}c,\tau'_{2}c\}$ $\tilde{t}=\{sc,s'c,s''c\}$
			
			 \end{center}

		 \par We can see that the condition $\{\tau_{1},\tau_{2}\},\{\tau'_{1},\tau'_{2}\}\subset\{s,s',s''\}$
		 is equivalent to the condition $\mathcal{ E },\mathcal{ E }'\subset\tilde{t}$.
		\item For every center $c$, we have $\mathcal{ E }=e(c,\tau)$ and every $\tau'$
		 s.t. $\tau'\sim^{L}\tau'$ induces a distinct $\mathcal{ E }'$. So, $d(\tau)=2\tilde{d}$
		 and $L$ is regular. And we have $d(\mathcal{ E })=|c(\mathcal{ E })|\tilde{d}$, where $|c(\mathcal{ E })|=2$ by \lemref{double}.
	\end{enumerate}
\end{proof}

\subsection{Replacement graph properties}
\label{repgraph}
\begin{defn}

    \label{def:We-define-graph}Given a commutative triplet structure $H=Cts(G,\mathcal{S})$,

	We define $G_{rep}$ that stands for a replacement product graph\footnote{The following are standard definitions. Some of them were taken from this excellent lecture about zig-zag product \cite{lect02} }.
	
	\[G_{rep}:=G_{cay}\raisebox{.5pt}{\textcircled{\raisebox{-.9pt}{r}}}L\]
	And more specifically, 	
	in $G_{rep}$, the vertex set is \[V_{rep}=G\times\mathcal{T}\]
	we define for \footnote{Notice that we need that the generators will commute here. The action is defined as in subsection  \ref{subsec:Structures} } $v\in G$, \[\phi_{v}:\mathcal{T}\rightarrow G\]
	\[
		\phi_{v}(\tau)=\tau v
	\]

	We have

	\[v=abu\text{ }\iff\text{ }(u,\{a,b\})\sim(v,\{a^{-1},b^{-1}\}) \]

	\medskip  \medskip and 
	 \[
	 E^{{\color{red}red}}=\{(v,\tau)\sim(v,\tau')\text{ if }\tau \sim \tau'\text{ on }L\}
	 \] 
	\[
		E^{{\color{blue}blue}}=\{(v,\tau)\sim(u,\tau')\ if\ u\sim v\ and\ \text{\ensuremath{\phi_{v}(\tau)=u\ ,\phi_{u}(\tau')=v\}}}
	\]


	\[
		E(G_{rep})=E^{{\color{blue}blue}}\cup E^{{\color{red}red}}
	\]

	$E^{{\color{blue}blue}}$ will be defined by the adjacency matrix $P_{{\color{blue}{\normalcolor }blue}}$ (or simply $P_B$)

	$E^{{\color{red}red}}$ will be defined by the adjacency matrix $P_{{\color{red}red}}$ (or $P_R$)

	$P_{R}|_{g}$ signifies restriction of $P_{R}$ to elements\footnote{That is $\{g,\tau \ \mid\ \tau\in \mathcal{T}\}$} $g,\_$ 
	. This would be of course exactly an instance of the graph $L$.
	So, $P_R$ is $\tilde{d}$-regular.

\end{defn}

\begin{rem*}
	Notice that
	the edges defined by $P_{R}P_{B}P_{R}$ are the edges of the zig-zag product(see \cite{hoory2006expander}).

    For an illustrative example for $G_{ cay }$, see corollary \ref{conlon}.
\end{rem*}

\subsubsection{Operator $T$}
\label{mainpart}

We define $T$ (will act over V($G_{rep}$)) by

\[
	T=\frac{1}{2}P_{R}+\frac{1}{2}P_{R}P_{B}
\]
(the matrices are normalized) 

We define an inverse function to $e:G\times\mathcal{T}\rightarrow E^{1}$ that is
\[\gamma:E^{1}\rightarrow P(G\times\mathcal{T})\]
Notice that $\gamma$ is the labeling function that gives the vertices
in $G_{walk}$ the corresponding name in $G_{rep}$. \newline

The random walk described by $T$ over $G_{rep}$ and the random walk over $G_{walk}$ are essentially the same,
as demonstrated in the following section:

\subsubsection{The random walk in Conlon's case}

\label{sec:randcon}
We describe the random walk in Conlon's case, in terms of types and centers. Suppose we start
at $e(g,\{s_{1},s_{2}\})$. In each step, we pick first a center
our edge is contained in. That is, we choose $k\in\{1,2\}$. So we will be 
either in center $g$ or in center $s_{1}s_{2}g$ in probability $1/2$. 
Now we choose another $s\in S$, where $s\neq s_{1},s_{2}$.

Then we look at the triple that contains the edge $\{s_{1}g,s_{2}g,sg\}$
for $s\in S$. The type of the new edge is $\{s,s_{j}\}$ for $j\in\{1,2\}$.

So the choices are exactly $2d$, where $d$ is the degree of $L$.
The graph $L$ is exactly $J(S,2)$ defined earlier, so $d$ is $2(S-2)$. 
There are $4(S-2)$ choices all in all.

Looking it as a random walk over $G_{rep}$, selecting a center corresponds to selecting either the first operand or second operand in $T$. 
And selecting the new edge type corresponds to an action by $P_R$. 

\subsubsection{Formal relation}
\label{formal}
In this section we prove that $T$ is a lift of $G_{walk}$. 

To prove it, we first prove there is a graph homomorphism between the graphs, and then we prove that the neighborhood are transfered bijectively.

\begin{lem}
 \label{lem:graphhomo}
 There is graph homomorphism between the graph induced by $T$ over $G_{rep}$ and $G_{walk}$. That is
 the function \[e:V(G_{rep})\rightarrow V(G_{walk})\] (defined earlier) such that if\footnote{There is a directed edge from $g,\tau$ to $g',\tau'$ or the other way around}
 \[g,\tau\rightarrow^{T}g',\tau'\] then \[e(g,\tau)\sim^{G_{walk}}e(g',\tau')\]
 This mapping is 2-1. 
\end{lem}

\begin{proof}

 Either\footnote{Because $P_B$ impose a condition on the center} \[e_{g',\tau'} \frac{1}{2} P_R e_{g,\tau}= \frac{1}{4\tilde{d}}\] or \[e_{g',\tau'} \frac{1}{2} P_R P_B e_{g,\tau}= \frac{1}{4\tilde{d}}\] 

In the first case, we have that $g=g'$ and $\tau \sim^L \tau' $. So we are done by lemma \ref{lem:InGwalk} for $c=g$.

In the second case, we have that $\tau g =g'$ and $ \tau^{-1} \sim^L \tau' $. So, by the same lemma, $e(g,\tau) \sim e(\tau g,\tau^{-1})$.
But as mentioned, $e(\tau g,\tau^{-1})=e(g,\tau)$. 

\end{proof}
We denote $\Gamma_{ \tilde{G}(v) }$ for the set of neighbors of $v$ in graph $\tilde{G}$.
\begin{lem}
 For every vertex $v\in V(G_{rep})$, the mapping
 \begin{align*}
     e : \Gamma_{T}(v) &\longrightarrow \Gamma_{G_{walk}}(e(v)) \\
 \end{align*}
 is bijective
\end{lem}
\begin{proof}
 The mapping is well-defined because it is a graph homomorphism. 
 It is enough to prove that the mapping is onto(because the sets are equal in size, every vertex in both graphs has $4\tilde{d}$ neighbors). 
 Suppose $\mathcal{E} \sim e(v)$ . 
 By lemma \ref{lem:InGwalk} we can assume that $e(v)=e(g,\tau)$ and $\mathcal{E}=e(g,\tau')$. 
 Therefore, $v$ is either $g,\tau$ or $\tau g,\tau^{-1}$.
 In any case, it is easy to see that $g,\tau' \in \Gamma_{G_{rep}}(v)$

\end{proof} 

The last lemma assured that the graph identifiable with $T$ is actually a \textbf{lift} of the graph $G_{walk}$. The definition of lift is the following (taken from \cite{hoory2006expander}):

\begin{defn}
    Let \(G\) and \(H\) be two graphs. We say that a function \(f: V(H) \rightarrow\)
    \(V(G)\) is a covering map if for every \(v \in V(H), f\) maps the neighbor set \(\Gamma_{H}(v)\) of \(v\) one-to-one and onto \(\Gamma_{G}(f(v))\). If there exists a covering function from $H$ to $G$,
we say that $H$ is a lift of $G$ or that $G$ is a quotient of $H$.
\end{defn}
It is well known that a lift has all the eigenvalues of the quotient(for example, see \cite{bilu2006lifts}).
We get as a conclusion that $T$ has all the eigenvalues of $G_{walk}$. So, we can deduce:
\begin{cor}
    \label{corlift}
    $ \lambda(G_{walk}) \le \lambda( T )$ 
\end{cor}
 
\subsubsection{Bounding the convergence rate} 

Now we want to get a bound on the convergence rate of the walk on $G_{walk}$ in terms of the walk on $T$. And then to get a bound on the random walk on $T$. \\

\noindent We define $\pi$ to be the uniform distribution on $V(G_{walk})$.\\
We define $\pi'$ to be the uniform distribution on all vertices of $V(G_{rep})$.\\

\noindent Notice that $\pi'$ on any vertex is half the value of $\pi$. 
Let $V_+$ be the space $x \bot \pi`$ where $x\in \mathbb{R}^{V(G_{rep})}$. We prove here that $T_+$ is well-defined\footnote{$T$ restricted to $V_+$}.

\begin{fact}
\label{fact: FactTPI}
The operator T satisfies the following:
\begin{enumerate}
 \item $T\pi'=\pi'$
 \item $T(V_+)\subset V_+$
\end{enumerate}
\end{fact}

\begin{proof}
\mbox{}\\*
 \begin{enumerate}
 \medskip \medskip 
	\item We have \[
	T\pi'=\frac{1}{2}P_{R}\pi'+\frac{1}{2}P_{R}P_{B}\pi'=P_{R} \pi'=\pi'
 \] \\ 
 \par Since $P_{B}\pi'=\pi'$. That is because $\forall x,\tau$ \\
\[
(P_{B}\pi')( e_{ x,\tau } )=\pi' ( e_{ \tau x,\tau^{-1} } )=\pi'( e_{ x,\tau } )
\]

\item Suppose $y\bot \pi'$. Then 

\begin{align*}
 \langle \pi',Ty \rangle &= \langle \pi', P_R y + P_RP_By \rangle \\
	 &= \langle \pi', P_R y \rangle + \langle \pi', P_R P_By \rangle = \langle P_R\pi',y \rangle + \langle P_R \pi', P_By \rangle = \\
	 &= \langle \pi', y \rangle + \langle \pi', P_B y \rangle = \langle P_B \pi', y \rangle = 0 
\end{align*}

 \end{enumerate}\end{proof}

\begin{lem}
	\label{lem:In-case-there}
	$||T^2||_+ \le\frac{1}{2}+\frac{1}{2}\lambda(G_{cay}\raisebox{.5pt}{\textcircled{\raisebox{-.9pt}{z}}}L )$
\end{lem}

\begin{proof}

	\[
		T=\frac{1}{2}P_{R}+\frac{1}{2}P_{R}P_{B}
	\]


	So, we have:

	\[
		T^{2}=\frac{1}{4}[P_{R}^{2}+P_{R}^{2}P_{B}+P_{R}P_{B}P_{R}+P_{R}P_{B}P_{R}P_{B}]
	\]
We will bound $||T^2||_+$ 

	\[
		||P_{R}^{2}+P_{R}^{2}P_{B}||_{+}\le2
	\]

	\[
		||P_{R}P_{B}P_{R}+P_{R}P_{B}P_{R}P_{B}||_{+}\le||P_{R}P_{B}P_{R}||_{+}+||P_{R}P_{B}P_{R}P_{B}||_{+}\le2||P_{R}P_{B}P_{R}||_{+}
	\]
	(Notice that $P_Bx \bot \pi'$ if $x\bot \pi'$)

	\par So we have that

	\[
	    ||T^{2}||_{+}\le\frac{1}{2}+\frac{1}{2} \lambda(G_{cay}\raisebox{.5pt}{\textcircled{\raisebox{-.9pt}{z}}L})  
	\]	
	That assures that we hare a rapid convergence — as the eigenvalues of $T_+^2$ are bounded away from 1.

\end{proof}

We have our main theorem: 

\mainthm

\begin{proof}
 Immediate from corollary \ref{corlift} and lemma \ref{lem:In-case-there} 
\end{proof}

\begin{cor}
  \label{corollary:rw}
  $\lambda(G_{walk})\le \sqrt{\frac{1}{2}+\frac{1}{2}f(\alpha,\beta)}$
 where $\alpha=\lambda(G_{cay})$, $\beta=\lambda(L)$
 and $f$ is the zig-zag function\footnote{originally defined in \cite[Theorem 3.2]{reingold2002entropy}} from \cref{equation:zig-zag}  ($f$ satisfies $f(a,b)\le a+b$, $f<1$ where
  $a,b<1$ ).
\end{cor}
\begin{proof}
We now wish to bound $\lambda(G_{cay}\raisebox{.5pt}{\textcircled{\raisebox{-.9pt}{z}}}L )$ to get a bound on $G_{walk}$. 
We can rely on known theorems about the zig-zag product.
We use here the following theorem by Reingoldn, Vadhan and Wigderson (originally \cite[Theorem 4.3]{reingold2002entropy}).
The theorem reads:
	\begin{thm} 
		\begin{math}
		\text {If } G_{1} \text { is an }\left(N_{1}, D_{1}, \lambda_{1}\right)\text{-graph and } G_{2} \text { is a }\left(D_{1}, D_{2}, \lambda_{2}\right)\text {-graph then }\end{math}
		
		$G_{1}\raisebox{.5pt}{\textcircled{\raisebox{-.9pt}{z}}}G_{2}$\begin{math}\text { is a }\left(N_{1} \cdot D_{1}, D_{2}^{2}, f\left(\lambda_{1}, \lambda_{2}\right)\right)-\text{graph, where } f\left(\lambda_{1}, \lambda_{2}\right) \leq\lambda_{1}+\lambda_{2} \text{ and } f\left(\lambda_{1}, \lambda_{2}\right)<1 \text{ when } \lambda_{1}, \lambda_{2}<1\end{math}.
	\end{thm}
$f$ is the function:
\begin{equation}
 \label{equation:zig-zag} 
f\left(\lambda_{1}, \lambda_{2}\right)=\frac{1}{2}\left(1-\lambda_{2}^{2}\right) \lambda_{1}+\frac{1}{2} \sqrt{\left(1-\lambda_{2}^{2}\right)^{2} \lambda_{1}^{2}+4 \lambda_{2}^{2}}
\end{equation}	 
	
We call it the "zig-zag function".

Some of its properties are studied there. It is better (lower) when $\lambda_1$ and $\lambda_2$ are worse. And less than 1 if $\lambda_1$ and $\lambda_2$ are less than 1. This assures the resulted graph is an expander when the original graphs are. 

    \end{proof}    
\begin{restatable}{cor}{comutativegood}
 \label{corollary:comgood}
Let $G$ be a group.
Let $\mathcal{S}\subset{G \choose 3}$ be a set of triples. \\
Suppose $H=Cts( G,\mathcal{S} )$ is a $k$-edge regular (where $k$ is bounded by $D$) commutative triplet structure (satisfies conditions \hyperref[itemZ]{\textbf{0}} - \hyperref[itemE]{\textbf{E}} ).
Suppose further that $Cay(G,{\mathcal{S}\choose{2}}) $ is $\epsilon$-expander.
Then, $2D$-random walk on $H$ converges rapidly with some rate $\alpha(D,\epsilon)<1$.
\end{restatable}

\begin{proof}
Since $H=Cts( G,\mathcal{S} )$ is of bounded-degree, the number of vertices of $L$ is bounded. By condition \hyperref[itemD]{\textbf{D}}, $L$ is connected. 
So, each graph $L$ has convergence rate less than 1. 
There are only finitely many possibilities, so there is a number $\beta'(D)<1$ that is the maximal convergence rate for all the graphs $L$. \\

\par We can use corollary \ref{corollary:rw} and get:
\[\lambda(G_{walk})^2 \le \frac{1}{2}+\frac{1}{2}f(1-\epsilon,\beta')<1\]

\end{proof}
\section{Applications} 
\subsection{Conlon's construction}
We bound the convergence rate of $2D$-random walk in case of Conlon's construction:
\conlons
\begin{proof}
    As noted in section \ref{ssub:conlon}, $H=Cts(G,\mathcal{S})$ is a commutative triplet structure.

In our case: 
	\[
		G_{cay}=Cay(G,S+S-\{0\})
	\]

	and
	\[
		G_{rep}=G_{cay}\raisebox{.5pt}{\textcircled{\raisebox{-.9pt}{r}}}J(S,2)
	\]

	We define $\mu$ to be the maximal non-trivial eigenvalue of the original adjacency matrix $A_{cay}$, and $\alpha$ would be the normalized version of it. \\
	Let $A$ be the adjacency matrix of $Cay(G,S)$. 
	Notice that 
	\[
	 A_{cay} = \frac{1}{2} (A^2-dI)
	\]
	The proof is in \cite[Lemma 2.1]{conlon2017hypergraph}.	Then
	\[
		\mu=\frac{1}{2}(\lambda^{2}d^2 -d)
	\]

	\[
	 |\alpha|=\frac{2}{d(d-1)}\frac{1}{2}|\lambda^{2}d^{2}-d|\le\lambda^{2} ( 1+\frac{1}{d+1} )
	\]

	By corollary \ref{corollary:rw} 
	\[
	 \lambda(G_{walk})^2 \le \frac{1}{2}+\frac{1}{2}f(\alpha,\beta)\le\frac{1}{2}(1+\alpha+\beta)\le\frac{1}{2}(\frac{3}{2}+\alpha)=\frac{3}{4}+\frac{1}{2}\alpha
	\]
	As we know that \[\lambda(J(S,2))=\beta\le\frac{1}{2}\] (we have no
	real hope of it being anything better, because of \cite[Proposition 3]{alon1994random}\footnote{suggesting that bounded-degree expander family of Cayley graphs of abliean groups doesn't exist})

	\[
		\sqrt{\frac{3}{4}+\frac{1}{2}\alpha}\le\sqrt{\frac{3}{4}}\sqrt{1+\frac{2}{3}\alpha}\le\frac{\sqrt{3}}{2} +\frac{1}{2\sqrt{3}}\alpha\sim0.867+0.289\alpha
	\]

	So we get convergence rate of 

	\[
	 \frac{\sqrt{3}}{2} +\frac{1}{2\sqrt{3}}\lambda^2 + O(\frac{1}{d}) 
	\]

\end{proof}
\subsection{3-product case}
First, let us recall the definition of the hypergraph.
\thrprodcase*
We can see that an $e\in E^{1}$ is given by $\{s_{i}g,s_{j}g\}$ for
a certain $i\neq j$ where $1\le i,j\le 3$. We view this edge as of center $g$ and type
$\tau=\{s_{i},s_{j}\}\in S_{i},S_{j}$. Will call $ij$ the template
of the edge. And two edges in the center of $g$ are in the same triplet
if they share a single element. That is, they are adjacent if their
types are adjacent in $J(S_{1}\cup S_{2}\cup S_{3},2)$ and are not
of the same template. 

That is also the description of $L$. It is the graph on types $\mathcal(T)$ that are adjacent in $J(S_{1}\cup S_{2}\cup S_{3},2)$ and are not
of the same template. 

First, we: verify that the required conditions are satisfied.
Condition \hyperref[itemZ]{ \textbf{0} },\hyperref[itemA]{ \textbf{A} },\hyperref[itemC]{ \textbf{C} } can be easily seen that obtained.
Condition \hyperref[itemC]{\textbf{C}} is satisfied by the fact that $S_i$ has the same size.
Condition \hyperref[itm:itemB]{\textbf{B}} is satisfied because of the following observation:

Assuming $t$ is of template $ij$ and $t'$ of $i'j'$ and \[t_2 t_1=t'_2 t'_1\] 
For $g=(g_1,g_2,g_3)$, we can see that $g'=t_2t_1g$ has one position that is unaffected by $ij$.
Since this position has to be the same in both case, $ij$ and $i'j'$ are the same template.

Then since $S_i$ and $S_j$ commute, we have \[t_2t_1=t_1t_2\] and this is the only case.


\medskip  \medskip If we define $a_{i}$ to be the $i$-th element of $S_{1}$ ($b_{i}$ for $S_{2}$ and $c_{i}$ for $S_{3}$). \\
An element $v\in G$ in the link of an element $u\in G$ is characterized
by \[s\in\cup_{i \neq j}S_{i}S_{j} \text{ s.t. }s \cdot u = v\] \\

So if we part the link into 3 parts\footnote{as was done in the original paper \cite{chapman2018expander}} \[V=V_{1}\cup V_{2}\cup V_{3}\]
where
\[V_{3}=\{a_{i},b_{j},e\ |1\le i,j\le d\}\] and similarly for $V_{1},V_{2}$ \\

Two elements in $V_{2}$ and $V_{3}$ are adjacent $(a_{i},e,c_{j})\sim(e,b_{i'},c_{j'})$
iff $j=j'$, and similarly for $V_{1},V_{2}$. \\ 

This graph is called $L$ in the paper by Chapman, Linal and Peled, but we called it here $L'$. 
The spectrum was already analyzed there.
It was found that $\lambda(L')=\frac{1}{2}$ by \cite[Lemma 4.1]{chapman2018expander}.

We notice that this graph is isomorphic to the graph $L$ we defined, by the mapping \begin{align*}
    \psi : V(L') &\longrightarrow V(L) \\
		(a_i,e,c_j) &\longmapsto       \{a_i , c_j\}  \\
		(e,b_i,c_j) &\longmapsto       \{b_i , c_j\} \\
		(a_i,b_j,e) &\longmapsto       \{a_i , b_j\} \\
\end{align*}
 which is clearly an isomorphism. 

\medskip  \medskip Thus, we can finally prove:

\thrprod*

\begin{proof}
As it is a commutative triplet structure, we can apply corollary \ref{corollary:rw}, and get that 
\[\lambda(G_{walk}) \le \sqrt{\frac{1}{2}+\frac{1}{2}f(\alpha,\beta)}\] 
where $\alpha=\lambda(G_{cay})$ and $\beta=\lambda(L)$. 

As $\beta$ was found out to be $\frac{1}{2}$, it remains to find $\alpha$. 

	\medskip  \medskip We describe $G_{cay}$ explicitly. 
	We have \[G_{cay}=Cay(G,S_{2}S_{3})\cup Cay(G,S_{1}S_{3})\cup Cay(G,S_{1}S_{2})\]
	\[A_{i}=Cay(G_{i},S_{i})\] 
	and so we need to sum the different adjacency matrices.
	\[
		A_{G_{cay}}=\frac{1}{3}[A_{1}\otimes A_{2}\otimes I+I\otimes A_{2}\otimes A_{3}+A_{1}\otimes I\otimes A_{3}]
	\]

	We want to find eigenvectors.\\ Assuming $x_{\mu},$ $y_{\nu},$ $z_{\xi}$
	with corresponding eigen values $\mu,\nu,\xi$ in $A_{1},A_{2},A_{3}$
\[
	A_{G_{cay}}(x_{\mu}\otimes y_{\nu}\otimes z_{\xi})=\frac{1}{3}[\mu\nu+\nu\xi+\mu\xi](x_{\mu}\otimes y_{\nu}\otimes z_{\xi})
\]

We have $n^{3}$ eigenvectors and so, this is complete base.
The maximal non-trivial eigenvector of $G_{cay}$ has one of the corresponding eigenvalues $\mu,\nu,\xi$ less than $1$. 

This would be the maximal. So we take 
\[v=x_{\lambda_{3}}\otimes y_{1} \otimes z_{1}\] as the eigenvector, and we get $\frac{1+2\lambda_{3}}{3}$ as the eigenvalue.
\end{proof}
\section{Concluding Remarks}
Recent HDE combinatorial constructions have all used graph product,
in order to achieve $2D$-random walk convergence of bounded degree 
\footnote{Actually, other than those, essentially all such constructions(that provide 2D-random walk convergence and of course, that we know of) are algebraic in nature}

We can see how the commutative relations leads to 4-cycles in the skeleton and eventually to good expansion. 
By viewing the random walk as mixing both inside centers and between centers(using the 4-cycles), we can get the required convergence. 

However, we are still a bit far from the estimated theoretical lower bound of $\nicefrac{1}{2}$ by \cite{kaufman2017high}\footnote{The result there that is $\nicefrac{2}{3} $ refers to a lazy random walk. We have corrected it to match a non-lazy walk. }.

It would be nice and potentially more powerful to base the construction upon non-trivial 4-cycles, possibly getting close to the theoretical limit.
We assume it could be handled similarly. One can also think about the generalized version of Conlon's structure, in which the underlying group is not abliean.
We know for instance that Ramanujan Complex contains many 4-cycles in its 1-skeleton. It would be nice to understand them in similar terms, if possible.

In any case, I hope this structure could encourage one to come up with better applications of it, providing better expansion, than the cases analyzed here. 

Some more remarks concerning things we could have done:

\begin{itemize}
	\item It is worth to mention that neglecting condition \hyperref[itemE]{\textbf{E}}, 
		results in a similar structure, that contains self-loop in types that their generators doesn't commute. 
		And is actually closed to zig-zag product. It can be analyzed similarly, and provide expansion(although less) for other applications. 
		We have left this out for brevity. But we intend to include some of it in the future. 

	\item Though it is not completely verified yet, condition \hyperref[itm:itemB]{\textbf{B}} seems to be redundant, as some other possibilities of 4-cycles 
		would only improve the convergence overall. However, we are not that sure about the bijective property that stems from it. 
		Once again, it complicates things, and we have left out our approach regarding dealing with this case outside the paper.

	 \item We are not sure that using Cayley graph is essential to the structure. Probably not. Once again, it simplifies things considerably. 

\end{itemize}

Possibly it could also be used for analysis of other features such as vertex expansion, geometric overlapping or cosystolic expansion.
We haven't dealt with any of these, as this is not the focus of the paper, but some result were obtained for both cases we studied in the corresponding papers \cite{conlon2017hypergraph} and \cite{chapman2018expander}. 

We suspect this method could provide bounded-degree expanders of higher dimensions too in the future. Some progress for conducting a similar analysis in 4 dimensions has been made, 
and we already have a conjecture concerning this case.

The conjecture is the following:
\begin{conjecture}
    Under the notion in \cite{chapman2018expander},
 Let \(G\) be a d-regular triangle-free graph with n vertices, where
 \(\lambda(G)=(1-\epsilon) d,\) and let \(\Gamma=\mathcal{C}_{G_{[1,1,0]}^{(3)} .} .\) Then $Aux(\Gamma)$ satisfies mixing of 3-dimensional walk as well,  which depends on $\epsilon$.
\end{conjecture}
Further details will follow.

\subsection{Acknowledgements}
We thank Nati Linal and Michael Chapman, for fruitful discussions.
We also thank David Conlon for an inspiring lecture, and some more inspiring ideas\footnote{some of them appear in \cite{conlon2018hypergraph}}.

\printbibliography

\end{document}